\documentclass[12 pt]{amsart}
  \usepackage[utf8]{inputenc}
  \usepackage[margin=3cm]{geometry}
\makeatletter
\def\l@subsection{\@tocline{1}{0pt}{2pc}{1pc}{}}
\def\l@subsubsection{\@tocline{2}{0pt}{2pc}{1pc}{}}
\makeatother

  \usepackage{stmaryrd}
\usepackage{amsmath,amssymb,amsthm, mathrsfs, mathtools, bbold, mathbbol}
\usepackage{amscd,mathtools}
\usepackage{comment}
  \usepackage{geometry}
  \usepackage{graphicx,epstopdf}
  \usepackage{color,xcolor}
  \usepackage{enumerate}
  \usepackage{xspace}
  \usepackage{tipa}
  \usepackage[user,titleref]{zref}
 \usepackage{soul}
 \usepackage{xcolor}
  \usepackage{epsfig}

\usepackage{tikz}
\usetikzlibrary{shapes.geometric}
\usetikzlibrary{angles}

\usepackage{caption, subcaption}
\usepackage{tikz-cd}
\usetikzlibrary{calc,decorations.pathreplacing}
\usepackage{tikz}
\usetikzlibrary{calc}
\usetikzlibrary{arrows}
\usetikzlibrary{decorations.pathreplacing}
\usetikzlibrary{intersections}

\usepackage{mathrsfs}
\usetikzlibrary{matrix}
\usetikzlibrary{positioning}
\DeclareMathAlphabet{\pazocal}{OMS}{zplm}{m}{n}
\tikzset{>=stealth}

  \usepackage{hyperref}

  \hypersetup{
    colorlinks=true,
    linkcolor=blue,
    filecolor=magenta,      
    urlcolor=cyan,
    citecolor=purple,
    pdftitle={Overleaf Example},
    pdfpagemode=FullScreen,
    }

\usepackage{float}

  \newcommand{\calN}{\mathcal{N}}

  \newcommand{\calU}{\mathcal{U}}

  \newcommand{\EE}{\mathbb{E}}

  \newcommand{\RR}{\mathbb{R}}

  \newcommand{\ZZ}{\mathbb{Z}}

  \newcommand{\bfa}{\textbf{a}}
  \newcommand{\bfb}{\textbf{b}}

  \newcommand{\gothic}{\mathfrak}

  \newcommand{\go}{{\gothic o}}

  \newtheorem{theorem}{Theorem}[section]
  \newtheorem{proposition}[theorem]{Proposition}
  
  \newtheorem{lemma}[theorem]{Lemma}

  \newtheorem{claim}[theorem]{Claim}
  
  \newtheorem*{claim*}{Claim}

  \newtheorem{introthm}{Theorem}

  \theoremstyle{definition}
  \newtheorem{definition}[theorem]{Definition}

  \newtheorem*{question*}{Question}
  \newtheorem*{answer*}{Answer}
  \newtheorem*{application*}{Application}

  \theoremstyle{remark}
  \newtheorem{remark}[theorem]{Remark}
  \newtheorem*{remark*}{Remark}
  
  \newcommand{\pka}{\partial_{\kappa}}

  \newcommand{\Teich}{{Teichm\"uller }} 
  \newcommand{\sC}{{\sf C}}   
  \newcommand{\sD}{{\sf D}}

  \newcommand{\sQ}{{\sf Q}}

  \renewcommand{\bfa}{{\sf a}}

  \newcommand{\mm}{{\sf m}}   
  \newcommand{\nn}{{\sf n}}

  \newcommand{\qq}{{\sf q}}   
  \newcommand{\rr}{{\sf r}}

\DeclareMathOperator{\diam}{diam}

  \newcommand{\param}{{\mathchoice{\mkern1mu\mbox{\raise2.2pt\hbox{$
  \centerdot$}}
  \mkern1mu}{\mkern1mu\mbox{\raise2.2pt\hbox{$\centerdot$}}\mkern1mu}{
  \mkern1.5mu\centerdot\mkern1.5mu}{\mkern1.5mu\centerdot\mkern1.5mu}}}

\DeclarePairedDelimiterX{\norm}[1]{\lvert}{\rvert}{#1}
\DeclarePairedDelimiterX{\Norm}[1]{\lVert}{\rVert}{#1}

  \newcommand{\ST}{\mathbin{\Big|}} 
  \newcommand{\from}{\colon\thinspace}

  \newcommand{\ga}{{\gamma}}

\newcommand{\p}{\partial}

\newcommand{\CAT}{\ensuremath{\operatorname{CAT}(0)}\xspace}         
    
\title[Sublinearly Morse geodesics and First Passage Percolation]{Sublinear Morse geodesics and First Passage Percolation}
\author{Sagnik Jana}
 \address{Department of Mathematics,  University of Tennessee at Knoxville, Knoxville, TN, USA}
\email{sjana1@vols.utk.edu}

  \author{Yulan Qing}
 \address{Department of Mathematics,  University of Tennessee at Knoxville, Knoxville, TN, USA}
 \email{yqing@utk.edu}

\begin{document}
\maketitle
\begin{abstract}
Given an infinite connected graph, a way to randomly perturb its metric is to assign random i.i.d. lengths to the edges of the graph. Assume that the graph is infinite and of bounded degree. Assume also strict positivity and finite expectation of the edge length distribution and existence of a sublinearly Morse bi-infinite geodesic line, we prove that  almost surely there exists a bi-infinite geodesic line. This generalizes a previous result of \cite{BT17} regarding Morse geodesics.  
\end{abstract}

\section{Introduction}
First passage percolation is a model of random perturbation of a given metric graph. In this paper, we shall restrict to the simplest model, where random i.i.d lengths are assigned to the edges of a fixed graph,  and analyse the existence of infinite geodesic lines. We refer to \cite{ADH15},\cite{GK12},\cite{Kesten} for background and references.

We recall here how first passage percolation is defined following \cite{BT17}. We consider a connected non-oriented graph $X$, whose set of vertices and edges are denoted by $V$ and $E$, respectively. For every function $\omega: E \rightarrow(0, \infty)$, we equip $V$ with the weighted graph metric $d_{\omega}$, where each edge $e$ has weight $\omega(e)$. In other words, for every $v_{1}, v_{2} \in V, d_{\omega}\left(v_{1}, v_{2}\right)$ is defined as the infimum over all path $\gamma=\left(e_{1}, \ldots, e_{m}\right)$ joining $v_{1}$ to $v_{2}$ of $|\gamma|_{\omega}:=\sum_{i=1}^{m} \omega\left(e_{i}\right)$. We note that the infimum is always realized as a consequence of the bounded degree assumption. The graph metric on $X$ corresponds to the case where $\omega =1$, in this case, we denote the metric as $d$. Our model consists in choosing independently at random the weights $\omega(e)$ according to the probability measure defined on the set of all weight functions $\omega$. We refer to section ~\ref{sec23} for formal details.  



Let $X$ be a simplicial graph. Recall that a path $\gamma=\left(e_{1}, \ldots, e_{n}\right)$ between two vertices $x, y$ is a sequence of consecutive edges joining $x$ to $y$. We write
\[x=\gamma(0), \ldots, \gamma(n)=y\]
to mean the set of vertices such that for all $0 \leq i<n, \gamma(i)$ and $\gamma(i+1)$ are joined by the edge $e_{i+1}$. For all $i<j$, we shall also denote by $\gamma([i, j])$ the sub path $\left(e_{i+1}, \ldots, e_{j}\right)$ joining $\gamma(i)$ to $\gamma(j)$. Let $X_\omega$ denotes the graph $X$ with edge lengths $\omega(e)$ for each edge $e$.\newline
A bi-infinite geodesic is a geodesic which is doubly infinite that is, its vertices are indexed by $\mathbb{Z}$. It was conjectured that under i.i.d lengths of edges, almost surely there are no bi-infinite geodesics in FPP on $\mathbb{Z}^2.$ See \cite{Kesten} for more details. For negatively curved spaces, the situation changes dramatically. In \cite{WW}, Wehr proved that almost surely there is no bi-infinite geodesic which lies entirely in the upper half Euclidean plane. 

On the other hand, \cite{BT17} established the existence of bi-infinite geodesic line in the affirmative if we assume the existence of a Morse quasi-geodesic line before the percolation. In this case the geometric property of a Morse quasi-geodesic line strongly informs how it behaves after the percolation. We are thus inspired by the proof and show in this paper that even weaker assumptions suffices. That is, we suppose instead there is a geodesic line whose Morse probably weakens as it extends to infinity. This class of quasi-geodesics were constructed recently by \cite{QRT1,QRT2} and named \emph{sublinearly Morse } quasi-geodesics. The construction roughly replaces bounded neighborhood with a sublinear neighborhood with respect to the distance to the origin. In this paper, we prove
\begin{introthm}\label{introthm1}
Let $X$ be an infinite connected graph with bounded degree. Assume $\EE \omega_e < \infty$ and $\nu(0) = 0$. If $X$ contains a sublinearly Morse bi-infinite quasi-geodesic line, then for almost every $\omega$, there is a bi-infinite geodesic ray in $X_\omega$.
\end{introthm}
The proof is analogous to that of \cite{BT17}. We first offer a new property of sublinearly Morse geodesic rays and lines, which we refer to as middle recurrence. The name comes from the analogous property used in \cite{BT17}. The property is first exhibited by \cite{DMS} with an improved proof by \cite{ADT}. Using the middle recurrence, we derive a pointed superlinear divergence property when then is used in the main proof. In an upcoming paper, we aim to study whether sublinearly Morse directions are preserved under first passage percolation.
\subsection*{History}
For the Cayley graph of $\ZZ^2$ with respect to standard generators, first passage percolation was introduced by Hammersley and Welsh \cite{HW65}.  Benjamini, Tessera, and Zeitouni \cite{BZ12, BT17} studied the question in Gromov hyperbolic spaces and a number of questions have been raised there. In particular, \cite{BZ12} established the tightness of fluctuations of the passage time from the center to the boundary of a large ball, and \cite{BT17} established the almost sure existence of bi-geodesics in hyperbolic spaces. Recently Basu-Mj \cite{BM22} studied the first passage percolation on Gromov hyperbolic group $G$ with boundary equipped with the Patterson-Sullivan measure $\nu$. They showed that the variance of the first passage time grows linearly with word distance along word geodesic rays in every fixed boundary direction.

The sublinearly Morse boundary is a  geometric boundary construction that generalize the Gromov boundary. It is shown to be always quasi-isometrically invariant and  metrizable, and is frequently a group-invariant topological model for suitable random walks on the associated group. In comparison, Morse boundary, constructed in \cite{Morse, contracting}, in comparison, often lacks metrizability and Poisson boundary realization property.   Meanwhile, the genericity of a more geometric flavor is also exhibited for sublinearly Morse boundaries. In \cite{GQR22}, genericity of sublinearly Morse directions under Patterson Sullivan measure was shown to hold in the more general context of actions which admit a strongly contracting element. In fact, the results in \cite{GQR22} concerning stationary measures were recently claimed in a different setting by Inhyeok Choi \cite{Choi}, who, in place of ergodic theoretic and boundary techniques uses a pivoting technique developed by Gou\"ezel\cite{Gou22}. Also, following  \cite{wenyuan}, genericity of sublinearly Morse directions on the horofunction boundary was shown for all proper statistically convex-cocompact actions on proper metric spaces \cite{QY24}. Most recently, Wen defined and studied the genericity sublinearly Morse directions in higher rank symmetric spaces \cite{rouwen25}.

\subsection*{Open questions}
An immediate question to ask is if the sublinearly Morse property can be detected on the bi-infinite geodesic line exhibited in this paper. The authors of this paper thus continue to study and aim to give a positive answer in an upcoming preprint \cite{JQ25b}. A harder and open question is to ask if first passage percolations studied in this paper preserve the Morse property of the bi-infinite geodesic line. 
\subsection*{Acknowledgement} The second-named author would like to thank Gabriel Pallier for suggesting the problem and for helpful discussions.

\section{preliminaries}

\subsection{Quasi-isometry and quasi-isometric embeddings}

\begin{definition}[Quasi Isometric embedding] \label{Def:Quasi-Isometry} 
Let $(X , d_X)$ and $(Y , d_Y)$ be metric spaces. For constants $q \geq 1$ and
$Q \geq 0$, we say a map $f \from X \to Y$ is a 
$(q, Q)$--\textit{quasi-isometric embedding} if, for all points $x_1, x_2 \in X$
$$
\frac{1}{q} d_X (x_1, x_2) - Q  \leq d_Y \big(f (x_1), f (x_2)\big) 
   \leq q \, d_X (x_1, x_2) + Q.
$$
If, in addition, every point in $Y$ lies in the $Q$--neighbourhood of the image of 
$f$, then $f$ is called a $(q, Q)$--quasi-isometry. When such a map exists, $X$ 
and $Y$ are said to be \textit{quasi-isometric}. 

A quasi-isometric embedding $f^{-1} \from Y \to X$ is called a \emph{quasi-inverse} of 
$f$ if for every $x \in X$, $d_X(x, f^{-1}f(x))$ is uniformly bounded above. 
In fact, after replacing $q$ and $Q$ with larger constants, we assume that 
$f^{-1}$ is also a $(q, Q)$--quasi-isometric embedding, 
\[
\forall x \in X \quad d_X\big(x, f^{-1}f(x)\big) \leq Q \qquad\text{and}\qquad
\forall y \in Y \quad d_Y\big(y, f\,f^{-1}(x)\big) \leq Q.
\]
\end{definition}

A \emph{geodesic ray} in $X$ is an isometric embedding $\beta \from [0, \infty) \to X$. We fix a base-point $\go \in X$ and always assume that $\beta(0) = \go$, that is, a geodesic ray is always assumed to start from this fixed base-point. 
\begin{definition}[Quasi-geodesics] \label{Def:Quadi-Geodesic} 
In this paper, a \emph{quasi-geodesic ray} is a continuous quasi-isometric 
embedding $\beta \from [0, \infty) \to X$  starting from the basepoint $\go$. 
\end{definition}
The additional assumption that quasi-geodesics are continuous is not necessary for the results in this paper to hold, but it is added for convenience and to make the exposition simpler. 

If $\beta \from [0,\infty) \to X$ is a $(\qq, \sQ)$--quasi-isometric embedding, and $f \from X \to Y$ is a $(q, Q)$--quasi-isometry then the composition  $f \circ \beta \from [t_{1}, t_{2}] \to Y$ is a quasi-isometric embedding, but it may not be continuous. However, one can adjust the map slightly to make it continuous (see Definition 2.2 \cite{QRT1}) such that $f \circ \beta$ is a $(q\qq, 2(q\qq + q \sQ + Q))$--quasi-geodesic ray.

Similar to above, a \emph{geodesic segment} is an isometric embedding 
$\beta \from [t_{1}, t_{2}] \to X$ and a \emph{quasi-geodesic segment} is a continuous 
quasi-isometric embedding \[\beta \from [t_{1}, t_{2}] \to X.\] 

\noindent \textbf{Notation}. In this paper we will use $\alpha, \beta,$ etc. to denote quasi-geodesic rays. If the quasi-geodesic constants are $(1, 0)$, we use $\alpha_{0}, \beta_{0},$etc. to signify that they are in fact geodesic rays.  Meanwhile, we use $[\alpha], [\beta],...$ to denote the $\kappa$-equivalence classes of quasi-geodesic rays (see Definition~\ref{Def:Fellow-Travel}), and we also use $\bfa, \bfb,$ etc. to denote $\kappa$-equivalence classes without referring an element in each class. In contrast, we use $\alpha(\infty)$ to denote equivalence classes of $\alpha$ in the visual boundary (see Chapter II.8 in \cite{BH1}).

Furthermore, let $\alpha$ be a (quasi-)geodesic ray $\alpha \from [0, \infty) \to X$, if $x_{1}, x_{2}$ are points on $\alpha$, then the segment of $\alpha$ between $x_{1}$ and $x_{2}$
is denoted $[x_{1}, x_{2}]_{\alpha}$. If a segment is presented without subscript, for example $[y_{1}, y_{2}]$, then it is a geodesic segment between the two points.
Let $\beta$ be a quasi-geodesic ray. Define 
\[
\Norm{x} : = d(\go, x).
\]
For $\rr>0$, let $t_\rr$ be the first time where $\Norm{\beta(t)}=\rr$ and define:
\begin{equation}\label{notation}
\beta_\rr := \beta(t_\rr)
\qquad\text{and}\qquad
\beta|_{\rr} : = \beta{[0,t_\rr]} = [\beta(0),  \beta_{\rr}]_{\beta}
\end{equation}
which are points and segments in $X$, respectively. 

\subsection{Sublinearly Morse geodesic lines} \hfill

Let $\kappa \from [0, \infty) \to [1, \infty)$ be a sublinear function that is monotone increasing and concave. That is
\[
\lim_{t \to \infty} \frac{\kappa(t)}{t} = 0. \label{subfunction}
\]

The assumption that $\kappa$ is increasing and concave makes certain arguments
cleaner, otherwise they are not really needed. One can always replace any 
sub-linear function $\kappa$, with another sub-linear function $\overline \kappa$
so that \[\kappa(t) \leq \overline \kappa(t) \leq \sC \, \kappa(t)\] for some constant $\sC$ 
and $\overline \kappa$ is monotone increasing and concave. For example, define 
\[
\overline \kappa(t) = \sup \Big\{ \lambda \kappa(u) + (1-\lambda) \kappa(v) \ST 
\ 0 \leq \lambda \leq 1, \ u, v>0, \ \text{and}\ \lambda u + (1-\lambda)v =t \Big\}.
\]
The requirement $\kappa(t) \geq 1$ is there to remove additive errors in the definition
of $\kappa$--contracting geodesics(See Definition~\ref{Def:generalContracting}).

\begin{definition}[$\kappa$--neighborhood]  \label{Def:Neighborhood} 
For a closed set $Z$ and a constant $\nn$ define the $(\kappa, \nn)$--neighbourhood of $Z$ to be 
\[
\calN_\kappa(Z, \nn) = \Big\{ x \in X \ST 
  d_X(x, Z) \leq  \nn \cdot \kappa(x)  \Big\}.
\]

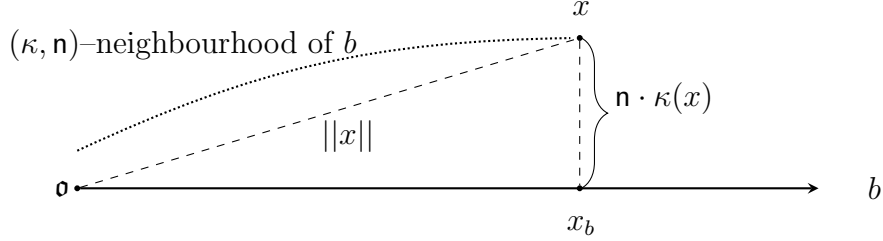
\begin{figure}[h]
\begin{tikzpicture}
 \tikzstyle{vertex} =[circle,draw,fill=black,thick, inner sep=0pt,minimum size=.5 mm] 
[thick, 
    scale=1,
    vertex/.style={circle,draw,fill=black,thick,
                   inner sep=0pt,minimum size= .5 mm},
                  
      trans/.style={thick,->, shorten >=6pt,shorten <=6pt,>=stealth},
   ]

 \node[vertex] (a) at (0,0) {};
 \node at (-0.2,0) {$\go$};
 \node (b) at (10, 0) {};
 \node at (10.6, 0) {$b$};
 \node (c) at (6.7, 2) {};
 \node[vertex] (d) at (6.68,2) {};
 \node at (6.7, 2.4){$x$};
 \node[vertex] (e) at (6.68,0) {};
 \node at (6.7, -0.5){$x_{b}$};
 \draw [-,dashed](d)--(e);
 \draw [-,dashed](a)--(d);
 \draw [decorate,decoration={brace,amplitude=10pt},xshift=0pt,yshift=0pt]
  (6.7,2) -- (6.7,0)  node [black,midway,xshift=0pt,yshift=0pt] {};

 \node at (7.8, 1.2){$\nn \cdot \kappa(x)$};
 \node at (3.6, 0.7){$||x||$};
 \draw [thick, ->](a)--(b);
 \path[thick, densely dotted](0,0.5) edge [bend left=12] (c);
\node at (1.4, 1.9){$(\kappa, \nn)$--neighbourhood of $b$};
\end{tikzpicture}
\caption{A $\kappa$-neighbourhood of a geodesic ray $b$ with multiplicative constant $\nn$.}
\end{figure}
\end{definition}

\begin{definition} [$\kappa$-Morse I, $\kappa$-Morse II] \label{D:k-morse} \label{Def:Morse} 
Let $Z \subseteq X$ be a closed set, and let $\kappa$ be a concave sublinear function. 
We say that $Z$ is \emph{$\kappa$-Morse} if one of the following equivalent (see Proposition 3.10 \cite{QRT2})  condition holds:
\begin{enumerate}

\item[I.] There exists a proper function 
$\mm_Z : \mathbb{R}^2 \to \mathbb{R}$ such that for any sublinear function $\kappa'$ 
and for any $r > 0$, there exists $R$ such that for any $(q, Q)$-quasi-geodesic ray $\beta$
with $\mm_Z(\qq, \sQ)$ small compared to $r$, if 
$$d_X(\beta_R, Z) \leq \kappa'(R)
\qquad\text{then}\qquad
\beta|_r \subset \calN_\kappa \big(Z, \mm_Z(q, Q)\big)$$

\item[II.] There is a function
\[
\mm'_Z \from \RR_+^2 \to \RR_+
\]
so that if $\beta \from [s,t] \to X$ is a $(\qq, \sQ)$--quasi-geodesic with end points 
on $Z$ then
\[
[s,t]_{\beta}  \subset \calN_{\kappa} \big(Z,  \mm'_Z(q, Q)\big). 
\]
\end{enumerate}
\end{definition} 

\begin{remark}
By taking the maximum function of $\mm_Z, \mm'_Z,$ we may and will always assume that both conditions hold for the same $\mm_Z$, which we refer to as the $\kappa$-Morse gauge. Further,
\begin{equation} 
\mm_Z(q, Q) \geq \max(q, Q). 
\end{equation} 
\end{remark}

\begin{definition} \label{weakprojection}
Let $(X, d_X)$ be a proper geodesic metric space and $Z \subseteq X$ a closed subset, and let 
$\kappa$ be a concave sublinear function. A map $\pi_{Z} \from X \to \mathcal{P}(Z)$ is a $\kappa$-\emph{projection}  if there exist constants $D_{1}, D_{2}$, depending only on $Z$ and $\kappa$, such that for any points $x \in X$ and $z \in Z$, 
\[
\diam_X(\{z \} \cup \pi_{Z}(x)) \leq D_{1} \cdot d_X(x, z) + D_{2} \cdot \kappa(||x||).
\]
\end{definition}

A $\kappa$-projection differs from a nearest point projection by a uniform multiplicative error and a sublinear additive error. In particular, the nearest point projection is a $\kappa$-projection. Indeed, for $z \in Z$ and $w \in \pi_Z(x)$, we have
\[
d(z, w) \leq d(z, x) + d(x, w) \leq 2d(z, x).
\]
\begin{definition}[$\kappa$-weakly contracting] \label{Def:generalContracting}
For a closed subspace $Z$ of a metric space $(X, d_X)$ and a $\kappa$-projection $\pi_Z$ onto $Z$, 
we say $Z$ is \emph{$\kappa$-weakly contracting} with respect to $\pi_Z$ if there are constants $C_{1}, C_{2}$, depending only on $Z$ and $\kappa$, such that, for every 
$x,y \in X$
\[
d_X(x, y) \leq C_{1} \cdot d_X( x, Z) \quad \Longrightarrow \quad
\diam_X \big( \pi_Z(x) \cup  \pi_Z(y)  \big) \leq c_{2} \cdot \kappa(||x||).
\]
\end{definition} 
In the special case that $\pi_Z$ is the nearest point projection and $C_1 = 1$, this property was called $\kappa$\emph{-contracting} in \cite{QRT1}. It was shown in \cite{QRT1} that, in the setting of CAT(0) spaces, this is stronger than the $\kappa$-Morse condition.
\begin{proposition}\cite[Proposition A.6.]{QRT2}\label{contracting}
Let $(X, \go)$ be a proper geodesic metric space with a fixed base point and let $\alpha$ be a quasi-geodesic
ray in $X$. Let $\pi$ be any $\kappa$-projection from $X$ to $\alpha$ in the sense of Definition~\ref{weakprojection}.
Then if $\alpha$ is $\kappa$-weakly Morse, then it is $\kappa'$-weakly contracting with respect to $\pi$ for some sublinear function $\kappa'$.
\end{proposition}

\begin{definition}[$\kappa$--equivalence classes in $\pka X$] \label{Def:Fellow-Travel}
Let $\beta$ and $\gamma$ be two quasi-geodesic rays in $X$. If $\beta$ is in some 
$\kappa$--neighborhood of $\gamma$ and $\gamma$ is in some 
$\kappa$--neighborhood of $\beta$, we say that $\beta$ and $\gamma$ 
\emph{$\kappa$--fellow travel} each other. This defines an equivalence
relation on the set of quasi-geodesic rays in $X$ (to obtain transitivity, one needs to change $\nn$ of the associated $(\kappa, \nn)$--neighborhood). 
\end{definition}

We denote the equivalence class 
that contains $\beta$ by $[\beta]$:
\begin{definition}[Sublinearly Morse boundary]
Let $\kappa$ be a sublinear function as specified in Section~\ref{subfunction} and let $X$ be a \CAT space.
\[\pka X : = \{ \text{ all } \kappa\text{-Morse quasi-geodesics } \} / \kappa\text{-fellow traveling}\]
\end{definition}

\begin{theorem} [\cite{QRT2}]
Let $X, Y$ be a proper metric space and let $\kappa$ be a sublinear function. The $\kappa$--boundaries of $X, Y$ are denoted $\pka X, \pka Y$. Any quasi-isometry from $X$ to $Y$ induces a homeomorphism between $\pka X$ and $\pka Y$.
\end{theorem}


\subsection{First passage percolation} \label{sec23}
 $(X,d)$ be a connected graph with vertex set $V$ and edge set $E$ with the graph distance metric $d$. For a fix vertex $\mathfrak{0}, d(\gamma, \mathfrak{o})$ denote the set distance between $\mathfrak{o}$ and the set of vertices $\{\gamma(0), \gamma(1), \ldots\}$. \newline 
 Fix a probability measure $\nu$ supported on  $[0, \infty)$ and consider the product probability space
\[
\Omega = [0, \infty)^E, \quad \mathbb{P} = \nu^{\otimes E}.
\]
A typical element of $(\Omega, \mathbb{P})$ will be denoted by $\omega = \{\omega(e)\}_{e\in E}$. We define i.i.d random variables $X_e:\Omega \rightarrow [0, \infty)$ as $ X_e(\omega) = \omega(e) $ with law $\nu$. Assume $\nu(\{0\})=0.$
Assigning  $\omega(e)$ as the length of edge $e$ defines a random metric (a priori, a pseudometric) on $X$, known as the first passage percolation (FPP) metric. The metric is defined as follows:
\begin{definition}

Let $\gamma = \{e_1, \dots, e_k\}$ be an edge path. For $\omega \in \Omega$, the $\omega$- length of $\gamma$ is defined as, 
\[
|\gamma|_{\omega}  = \sum_{e \in \gamma} \omega(e)
\]
The FPP distance is defined by, 
\[
d_\omega(x, y) = \inf_\gamma |\gamma|_\omega
\]
where the infimum is over all paths $\gamma$ with terminal vertices $x$ and $y$.
    
\end{definition}
A path realizing $d_\omega(x, y)$ is called a \emph{\( \omega \)}-geodesic. 
\begin{proposition}\label{prop2.13}     \cite[Lemma 2.3, Lemma 2.5]{BT17} 
\begin{enumerate}
\item Let $X$ be a connected graph, and let $\gamma$ be a self avoiding path. Assume $0<b=\mathbb{E} \omega_{e}<\infty$. Then for a.e. $\omega$, there exists $r_{0}=r_{0}(\omega)$ such that for all $i \leq 0 \leq j$,

$$
|\gamma([i, j])|_{\omega} \leq 2 b(j-i)+r_{0}
$$
\item \label{extension} Let $X$ be an infinite connected graph with bounded degree, and let $o$ be some vertex of $X$. Assume $\nu(\{0\})=0$. Then there exists $c>0$ such that for a.e. $\omega$, there exists $r_{1}=r_{1}(\omega)$ such that for all finite path $\gamma$ such that $d(\gamma, o) \leq |\gamma|$, one has

\[|\gamma|_{\omega} \geq c|\gamma|-r_{1}\]
\end{enumerate}
\end{proposition}
\section{middle-recurrence of sublinear lines}
The idea of middle-recurrence as a characterization for Morse geodesics was first introduced by Drutu-Mozes-Sapir in \cite{DMS} and then further studied by \cite{ADT}. The characterization smartly circumvents the need for quasi-geodesic segments which renders it ideal for the application of first passage percolations. In this section, we follow the general idea of the proof of these papers and generalize the result analogously to sublinearly Morse quasi-geodesic rays.
\begin{definition}
A path in a metric space $X$ is a continuous function $p: I\rightarrow X$ from an interval $I$ to $X$. A simple path is a path that is injective. We sometimes call it a self-avoiding path in a graph, as it never visits the same vertex twice. \newline
If $p$ is any finite path in $X$, let $|p|$ denote the distance between its endpoints, and let $\ell(p)$ denote its arc length. Then the slope of $p$ is defined as:
\[
sl(p) = \frac{\ell(p)}{|p|}
\]

\end{definition}
\begin{definition}

 For a path $\gamma$, if $a,b \in \gamma$ then  $\gamma_{[a,b]}$ the set of $x \in \gamma$ lying between $a$ and $b$. The middle third of $\gamma_{[a,b]}$ is defined as follows: \[ \gamma_{\frac{1}{3}[a,b]} := \biggl\{x\in \gamma : \min \{d(x,a),d(x,b)\} \ge \frac{1}{3}\cdot d(a,b)\biggr\}\]

\end{definition}
\begin{definition}[$\frac{1}{3}$-middle recurrence]
    We say that a path (geodesic line or quasi-geodesic) $\beta$ is $\frac{1}{3}$-middle recurrent if for every $C\geq1$ there is a constant $c$ (depends on $C$ and $\beta$) such that the following holds:
    for any path $p$ with endpoints $a,b \in \beta$ satisfying $\ell(p)\leq Cd(a,b)$, intersects the $c\cdot \kappa$-neighborhood of the middle third of $\beta[a,b]$ for some sublinear function $\kappa$. That is,
\[ p \cap \calN_{\kappa} (\beta_{\frac 13[a,b]}, c )\neq \emptyset.\]
\end{definition}
\begin{lemma}\cite[Lemma 3.3]{ADT}\label{lemma33}
Let $\beta$ be a $\kappa'$–radius–contracting quasi-geodesic and suppose that $p$ is a path (of slope $D$) at distance
at least $K$  from $\beta$ whose endpoints have distance exactly $K$ from $\beta$. Let $|p|$ denote the distance between the endpoints of $p$. Let $s, e$ be the endpoints of $p$. Then
\[ \frac{\kappa'(K)}{2K} \geq \frac{1-(2K+\kappa'(K))/d(s, e)}{sl(p)}\]
\end{lemma}

\begin{proposition}[$\kappa$-Morse implies middle recurrence] \label{prop3.5}
Let $\beta$ be a bi-infinite geodesic line. Assume that  $\beta$ is $\kappa$-Morse with Morse gauge $m_\beta(q, Q)$. Let $p$ be any simple path with endpoints $a, b \in \beta$ and such that $\ell(p) < C d(a, b)$. Then there exists a $\kappa'$-sublinear neighborhood of $\beta$, where $\kappa'$ is the contracting function of $\beta$, and whose multiplicative constant $c(C, \beta)$ depending only on $\beta$ and $C$ such that 
\[ p \cap \mathcal{N}_{\kappa'} (\beta_{\frac 13[a,b]}, c(C, \beta)) \neq \emptyset.\]

\end{proposition}
\begin{proof}
    To proceed, let 
    $\kappa'$ denote the contracting function of $\beta$ of Proposition~\ref{contracting}. By way of contradiction, suppose there does not exist any sublinear neighborhood of $\beta$ in which $p$ intersects the middle third. We first define a \emph{linear neighborhood} with constant $c$:
     
  \[Ln(\beta, c) : = \{ x \in X | \, d(x, \beta) \leq c \cdot d(\go, x)\}.\]
  The assumption implies that there exists a constant $c_1(C)$ such that there exists a family of paths $\{p_i \}$ with slope at most $C$ such that for arbitrarily large $t$, there exists a $p_i$ and $t$ such that  $p_i(t)$ is outside of the $c_1$-linear neighborhood of the middle third of the geodesic segment connecting the endpoints of $p_i$ (which we denote $s_i$ and $e_i$). Without loss of generality we can assume $0<c_1 <1.$ Moreover, we can assume that 
  \[d(s_i, \go) \geq i.\]
  Since the path $p_i$ is chosen to traverse outside of the $Ln(\beta, c_1),$ we can assume without loss of generality that eventually there exists a constant $c_3$ such that 
  \[
  d(e_i, \go) \geq c_3 d(s_i, \go).
  \]

  Given any such $p_i$, let $p_i'$ be a connected segment of $p_i$ that is outside the $c_1$-linear neighborhood of $\beta$ and does not intersect the $c_1$-linear neighborhood of the middle third. The endpoints of $p'_i$ we denote $s'_i$ and $e'_i$.
  
  Since $\beta$ is $\kappa'$-contracting with respect to distance from the base point. For any ball $B$ disjoint from $\beta$ and centered at a point $x \in p'_i$, 
  \begin{equation}\label{contractingball}
  \diam(\pi_\beta(B)) < c_2 \kappa'(d(x, \go)).
  \end{equation}
  Since $x \in p'_i$ is outside of the $c_1$-linear neighborhood, we have that
  \[ c_1 d(x, \go) \leq d(x, \beta)  \]
  thus we have
  \[\diam(\pi_\beta(B)) < c_2 \kappa'(d(x, \go))
\leq c_2 \kappa'(\frac{1} {c_1} d(x, \beta))\leq \frac{c_2}{c_1} \kappa'(d(x, \beta)).\]
Thus, such balls are $\kappa'$-radius-contracting with constant $c_2/c_1$. Thus we have proven the following claim:
\begin{claim}\label{radial-contracting}
Let $\beta$ be a sublinearly Morse geodesic line that is 
$\kappa'$-contracting in the sense of Definition~\ref{Def:generalContracting} with constant $c_2$. Then any ball disjoint from $\beta$ and whose center is outside of the $c_1$-linear-neighborhood of $\beta$ has $\kappa'$-radius-contracting property with constant $\frac{c_2}{c_1}$.
\end{claim}

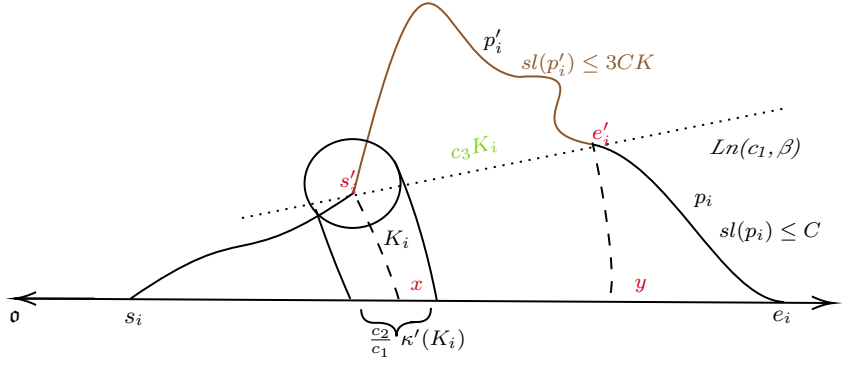
\begin{figure}
    \centering
\tikzset{every picture/.style={line width=0.75pt}} 

\begin{tikzpicture}[x=0.75pt,y=0.75pt,yscale=-1,xscale=1]

\draw    (50.24,164.76) -- (11.86,164.76) ;
\draw [shift={(9.86,164.76)}, rotate = 360] [color={rgb, 255:red, 0; green, 0; blue, 0 }  ][line width=0.75]    (10.93,-3.29) .. controls (6.95,-1.4) and (3.31,-0.3) .. (0,0) .. controls (3.31,0.3) and (6.95,1.4) .. (10.93,3.29)   ;
\draw  [dash pattern={on 0.84pt off 2.51pt}]  (124.62,124.12) -- (399.75,68.51) ;
\draw [color={rgb, 255:red, 139; green, 87; blue, 42 }  ,draw opacity=1 ]   (264.25,52.92) .. controls (312.43,49.39) and (255.73,79.26) .. (301.61,87.02) ;
\draw [color={rgb, 255:red, 139; green, 87; blue, 42 }  ,draw opacity=1 ]   (180.9,111.49) .. controls (222.96,-57.87) and (221.43,47.98) .. (264.25,52.92) ;
\draw    (301.61,87.02) .. controls (339.84,94.78) and (366.72,165.12) .. (398.07,166.53) ;
\draw    (68.49,164.76) .. controls (129.66,122.42) and (119.72,153.83) .. (180.9,111.49) ;
\draw  [draw opacity=0][dash pattern={on 4.5pt off 4.5pt}] (203.67,165) .. controls (201.2,156.72) and (197.36,146.41) .. (192.37,135.1) .. controls (188.42,126.14) and (184.23,117.61) .. (180.13,110.06) -- (181.27,132.81) -- cycle ; \draw  [dash pattern={on 4.5pt off 4.5pt}] (203.67,165) .. controls (201.2,156.72) and (197.36,146.41) .. (192.37,135.1) .. controls (188.42,126.14) and (184.23,117.61) .. (180.13,110.06) ;  
\draw  [draw opacity=0][dash pattern={on 4.5pt off 4.5pt}] (310.53,161.91) .. controls (312.21,155.03) and (310.32,133) .. (305.43,106.07) .. controls (304.16,99.08) and (302.79,92.31) .. (301.37,85.96) -- (297.56,103.19) -- cycle ; \draw  [dash pattern={on 4.5pt off 4.5pt}] (310.53,161.91) .. controls (312.21,155.03) and (310.32,133) .. (305.43,106.07) .. controls (304.16,99.08) and (302.79,92.31) .. (301.37,85.96) ;  
\draw   (156.18,106.7) .. controls (156.18,94.32) and (167.02,84.28) .. (180.39,84.28) .. controls (193.76,84.28) and (204.6,94.32) .. (204.6,106.7) .. controls (204.6,119.09) and (193.76,129.13) .. (180.39,129.13) .. controls (167.02,129.13) and (156.18,119.09) .. (156.18,106.7) -- cycle ;
\draw   (222.97,166.06) .. controls (217.86,139.31) and (210.72,115.71) .. (201.54,95.27) ;
\draw   (161.81,119.68) .. controls (167.14,135.68) and (172.96,150.75) .. (179.28,164.92) ;
\draw   (184.72,170.17) .. controls (184.81,174.84) and (187.19,177.12) .. (191.86,177.03) -- (192.45,177.02) .. controls (199.12,176.89) and (202.5,179.15) .. (202.59,183.82) .. controls (202.5,179.15) and (205.78,176.75) .. (212.45,176.62)(209.45,176.68) -- (213.04,176.61) .. controls (217.71,176.52) and (219.99,174.14) .. (219.9,169.47) ;
\draw    (9.86,164.76) -- (421,167.01) ;
\draw [shift={(423,167.02)}, rotate = 180.31] [color={rgb, 255:red, 0; green, 0; blue, 0 }  ][line width=0.75]    (10.93,-3.29) .. controls (6.95,-1.4) and (3.31,-0.3) .. (0,0) .. controls (3.31,0.3) and (6.95,1.4) .. (10.93,3.29)   ;

\draw (63.75,170.25) node [anchor=north west][inner sep=0.75pt]  [font=\scriptsize]  {$s_{i}$};
\draw (390.74,169.26) node [anchor=north west][inner sep=0.75pt]  [font=\scriptsize]  {$e_{i}$};
\draw (172.34,98.66) node [anchor=north west][inner sep=0.75pt]  [font=\tiny,color={rgb, 255:red, 208; green, 2; blue, 27 }  ,opacity=1 ]  {$s_{i} '$};
\draw (299.79,74.43) node [anchor=north west][inner sep=0.75pt]  [font=\tiny,color={rgb, 255:red, 208; green, 2; blue, 27 }  ,opacity=1 ]  {$e_{i} '$};
\draw (208.1,154.53) node [anchor=north west][inner sep=0.75pt]  [font=\tiny,color={rgb, 255:red, 208; green, 2; blue, 27 }  ,opacity=1 ]  {$x$};
\draw (321.3,153.21) node [anchor=north west][inner sep=0.75pt]  [font=\tiny,color={rgb, 255:red, 208; green, 2; blue, 27 }  ,opacity=1 ]  {$y$};
\draw (5.63,169.15) node [anchor=north west][inner sep=0.75pt]  [font=\scriptsize]  {$\mathfrak{o} $};
\draw (194.36,128.86) node [anchor=north west][inner sep=0.75pt]  [font=\tiny]  {$K_{i}$};
\draw (186.74,178.42) node [anchor=north west][inner sep=0.75pt]  [font=\tiny,xslant=0.01]  {$\frac{c_{2}}{c_{1}} \ \kappa '( K_{i})$};
\draw (227.09,86.1) node [anchor=north west][inner sep=0.75pt]  [font=\tiny,color={rgb, 255:red, 126; green, 211; blue, 33 }  ,opacity=1 ,rotate=-347.06]  {$c_{3} K_{i}$};
\draw (245.14,27.53) node [anchor=north west][inner sep=0.75pt]  [font=\tiny]  {$p_{i} '$};
\draw (361.35,82.19) node [anchor=north west][inner sep=0.75pt]  [font=\tiny,xslant=0.27]  {$Ln( c_{1} ,\beta )$};
\draw (263.95,39.38) node [anchor=north west][inner sep=0.75pt]  [font=\tiny,color={rgb, 255:red, 139; green, 87; blue, 42 }  ,opacity=1 ]  {$sl( p_{i} ') \leq 3CK$};
\draw (351.28,110.42) node [anchor=north west][inner sep=0.75pt]  [font=\tiny]  {$p_{i}$};
\draw (363.73,124.06) node [anchor=north west][inner sep=0.75pt]  [font=\tiny]  {$sl( p_{i}) \leq C$};

\end{tikzpicture}
  \caption{Behavior of the path outside $c_1$-linear neighborhood of $\kappa$-Morse geodesic line $\beta$.}
    \label{fig:enter-label}
\end{figure}
Secondly, suppose the end points of $p_i'$ are $(s'_i, e'_i)$. 
  and we know that by construction \[
  d(s_i, e_i) = c_3 d(s_i, \go).
  \]
  By construction there exists points $x, y \in \beta$ such that
  \[d(x, s_i') = c_1 d(\go, s_i') \,\text{ and }\, d(y, e_i') = c_1 d(\go, e_i')\]
  Now estimate,
 \begin{align*}
 d(\go,s_i')
 &\leq d(x,\go)+d(x,s_i')\\
 &\leq d(s_i,\go)+\frac{c_3-1}{3}d(s_i,\go)+d(s_i',\go)
 \end{align*} implies,
 \[(1-c_1)d(\go,s_i')\leq  \bigg(\frac{c_3-1}{3} +1\bigg)d(s_i, \go)\]
 Hence,
 \begin{equation}\label{estimate1}
   d(x,e_i')\leq \frac{c_1}{1-c_1} \bigg(\frac{c_3-1}{3} +1\bigg)d(s_i, \go)
   \end{equation}
Similarly we have,
\begin{equation}\label{estimate2}
 d(y,e_i')\leq \frac{c_1}{1-c_1} \bigg(\frac{2(c_3-1)}{3} +1\bigg)d(s_i, \go)   
\end{equation}   
 By the generalized triangle inequality and using estimate~\ref{estimate1} and ~\ref{estimate2} we have, 
 \begin{align*}
   d(s'_i, e'_i) 
   &\geq d(x, y) -\frac{c_1}{1-c_1} \bigg(\frac{c_3-1}{3} +1\bigg)d(s_i, \go)- \frac{c_1}{1-c_1} \bigg(\frac{2(c_3-1)}{3} +1\bigg)d(s_i, \go)\\
   & \geq \frac{1}{3}d(s_i,e_i)-\frac{c_1}{1-c_1}\frac{c_3+1}{c_3-1}d(s_i,e_i)\\
   & \geq \frac{1}{3} \cdot \frac{1}{K}\cdot d(s_i,e_i), \ \  \ \text{where}\ \frac{1}{K} :=1- \frac{c_1}{3(1-c_1)}\frac{c_3+1}{c_3-1}.
 \end{align*}
 We can choose $c_3$ large enough compared to $c_1$ so that $K$ is uniformly bounded above. 
  Thus 
  \[ sl(p'_i) = \frac{|p'_i|}{d(s'_i, e'_i)} \leq \frac{|p'_i|}{\frac{1}{3K}d(s_i, e_i)} \leq
   \frac{|p_i|}{\frac{1}{3K}d(s_i, e_i)} = 3KC.\]

Furthermore, we claim that for large enough $i$, 
  $d(s_i, e_i)$ does not grow sublinearly with respect to 
  $d( \go, s_i)$: otherwise, since slope of $p_i$ are bounded, eventually $p_i$ will not be outside of the linear neighborhood of $\beta$, thus there exists a constant $c_2>0$ where
  \[
   d(s_i, e_i) \geq c_2 d(\go, s_i).
  \]
  Let $K_i:= d(s'_i, \beta)$. A slight adaptation of the Lemma~\ref{lemma33} yields 
   \[ \frac{\kappa'(K_i)}{K_i} \geq \frac{6c_1C-(2 K_i +c_2'\kappa'(K_i))/d(s'_i, e'_i)}{c_2'(3KC+1)}\]
   Where $c_2'=\frac{c_2}{c_1}$ from claim ~\ref{contractingball}.\newline
   As we argued before,  as $i \to \infty$,  $d(s'_i, e'_i)$ grows linearly with respect to $d(s_i, e_i)$ and hence $K_i$. Thus by construction, there exists $c_4$ such that for all $i$ large enough, $d(s'_i, e'_i) \geq c_4 K_i$ and since $\lim_{t\to\infty}\frac{\kappa'(t)}{t}=0$, there exists a uniform constant $D>0$ depending only on $C, c_1, c_2, c_3,c_4, \kappa, \kappa'$ where 
  \[
  \frac{6c_1C-(2 K_i +c_2'\kappa'(K_i))/c_4K_i}{c_2'(3KC+1)}\geq D.
  \]

  Thus, we have that
  \[\frac{\kappa'(K_i)}{K_i} \geq D
\]
  There exists a maximal $K_i$ for which the inequality holds, which is a contradiction as desired. 

\end{proof}

\begin{lemma} \label{lemma36}   
Let $X$ be an infinite connected graph with bounded degree. Assume $\gamma_{0}$ is a bi-infinite $\kappa$-Morse quasi-geodesic. Then there exists an increasing function $\phi: \mathbb{R}_{+} \rightarrow \mathbb{R}_{+}$such that $\lim _{t \rightarrow \infty} \phi(t)=\infty$, and with the following property. Assume

\begin{itemize}
  \item $x, y$ belong to $\gamma_{0}$.
  \item $x'$ and $y'$ are vertices such that $d\left(x, x'\right)=d\left(y, y'\right)=R$, and $d\left(x', y'\right) \geq 4 R$;
  \item $\gamma$ is a quasi-geodesic joining $x'$ to $y'$, and remains outside of the $R$-neighborhood of $\gamma_{0}$.
\end{itemize}
Then
\[|\gamma| \geq \phi(R) d(x, y)\]
\end{lemma}

\begin{proof} Assume by contradiction that there does not exist such a $\phi(R)$, that is, eventually for the paths $\gamma$ long enough becomes a path whose length is bounded above linearly. Then Proposition~\ref{prop3.5} implies that these paths intersect a sublinear neighborhood of the middle third with bounded constants. However, the choice of $x, y$ is such that $\gamma$ is outside of larger and larger linear neighborhoods, thus we have the desired contradiction. Now we start with the formal proof.

Suppose by contradiction that there exists a constant $C>0$, and for every $n$, consider the path $\gamma$ joining the vertices $x',y'$ with $|\gamma| \leq C\cdot d(x,y)$,  such that $d\left(x, x'\right)=d\left(y, y'\right)=R, d\left(x', y'\right) \geq 4 R$, and $\gamma$ avoids the $R$-neighborhood of $\gamma_{0}$, where $R$ is an integer greater than $d(x,y)$. By our assumption, \[4R \leq d(x',y') \leq d(x',x)+d(x,y)+d(y,y')\leq 2R + d(x,y)\]  thus, \[d(x,y) \geq 2R\] Now consider the path $p$ that is a concatenation of  $[x,x']$, $\gamma$ and $[y',y]$, then $|p| \leq 2R+ |\gamma| \leq d(x,y) + C d(x,y) \leq C'd(x,y)$, where $C'=1+C$. Then by Proposition~\ref{prop3.5} there exists a $\kappa'$-sublinear neighborhood of $\gamma_0$ where $\kappa'$ is the contracting function of $\gamma_0$ such that $p$ intersects the $c \cdot \kappa'$-neighborhood of the middle third of $\gamma_0[x,y]$, where $c$ is the multiplicative constant depends only on $\gamma_0$ and $C'.$ Suppose $p_i$ be a point which is in the intersection. Since $d(x,y)\geq R$, by choosing $d(x,y)$ large enough we can assume also $R\geq c \cdot \kappa'(d(\mathfrak{o},p_i))$. This contradicts the assumption of Lemma ~\ref{lemma36} since $\gamma$ avoids $R$- neighborhood of $\gamma_0$. 
\end{proof}

Now we are ready to prove Theorem~\ref{introthm1}:
\begin{theorem}\label{thm38}
Let $X$ be an infinite connected graph with bounded degree. Assume $\EE \omega_e < \infty$ and $\nu(0) = 0$. If $X$ contains a sublinearly Morse bi-infinite quasi-geodesic line, then for almost every $\omega$, there is a bi-infinite geodesic ray $\gamma_\omega \subset X_\omega$. Furthermore, there exists a constant $K$ depending only on $\gamma_0 \subset X$ and $\omega$, such that
\[d_{X_\omega} (\go, \gamma_\omega) \leq K. \]
\end{theorem}
\begin{proof}
Let $\gamma_0$ be a bi-infinite $\kappa$-Morse quasi-geodesic line. By \cite[Lemma 4.2 (1)]{QRT2}, there exists a $\kappa$-Morse bi-infinite geodesic line in the neighborhood $\calN_\kappa( \gamma,m_\gamma(1,0))$. Thus without loss of generality, we can pass to a bi-infinite $\kappa$-Morse geodesic line, which we denote $\gamma_0$. Denote $\go=:\gamma_{0}(0)$. Consider two sequences of vertices $(x_{n})$ and $(y_{n})$ on $\gamma_{0}$ escaping to infinity in opposite directions such that 
\[d(x_n, \go) = n \text{ and } d(y_n, \go) = n.\]

Let $\Omega' \subset \Omega$ be a measurable subset of full measure for which Proposition ~\ref{prop2.13} holds. For all $n$ and for all $\omega$, we pick measurably an $\omega$-geodesic $\gamma_{\omega}^{n} \in X_{\omega}$ between $x_{n}$ and $y_{n}$ (these vertices are in both $X$ and $X_\omega$). 
By Proposition ~\ref{prop2.13}, we have that $d_{\omega}\left(x_{n}, y_{n}\right)$ is finite, so paths between $x_n, y_n$ of finite length whose $\omega$-length is near to $d_\omega(x_n,y_n)$. Any path between $x_n, y_n$ in $X$ has the property $d(\mathfrak{o, \gamma)\leq |\gamma|}$. Therefore by Proposition ~\ref{prop2.13} paths of length $\geq M$ have $\omega$-length going to infinity as $M \rightarrow \infty$. That is, extremely large paths must have large $\omega$-length. Therefore, the infimum of $\omega$-lengths of all such paths exists. Thus geodesic segments $\gamma_{\omega}^{n}$ exists. Furthermore,  let $\gamma^n$ denote the image of $\gamma_{\omega}^{n}$ in $X$. In fact, for an integer $i$, $\gamma^n(i)=\gamma_{\omega}^{n}(i).$
\newline
The proof has two parts, firstly, if we can prove that for all $\omega \in \Omega'$, there exists a constant $R_{\omega}>0$ such that for all $n, d\left(\gamma_{\omega}^{n}, \mathfrak{o}\right) \leq R_{\omega}$, then the main theorem follows by Arzelà–Ascoli theorem. So we shall assume by contradiction that for some $\omega \in \Omega'$, there exists a sequence $R_{n}$ going to infinity such that none of the vertices on the geodesic $\ga_n^\omega$ lies with in $10R_n$ graph distance from the base point $\mathfrak{0}$. 
\begin{lemma}\label{lemma38}
Assuming the above, there exist integers $p<q$ such that
\[d\left(\gamma_{\omega}^{n}(p), \gamma_{0}\right)=d\left(\gamma_{\omega}^{n}(q), \gamma_{0}\right)=R_{n}\]
and such that for all $p \leq k \leq q$,
\[
d\left(\gamma_{\omega}^{n}(k), \gamma_{0}\right) \geq R_{n}
\]and
\[
d\left(\gamma_{\omega}^{n}(p), \gamma_{\omega}^{n}(q)\right) \geq 4 R_{n}\]
\end{lemma}
\begin{proof}  Let $\gamma_{0}(i)$ and $\gamma_{0}(j)$ with $i<0<j$ be the two points at distance $10 R_{n}$ from $o=\gamma_{0}(0)$. Since $\gamma_{0}$ is a $\kappa$-Morse geodesic line, $\gamma_{0}((\infty, i])$ and $\gamma_{0}([j, \infty))$ are distance $20 R_{n}$ from one another. 

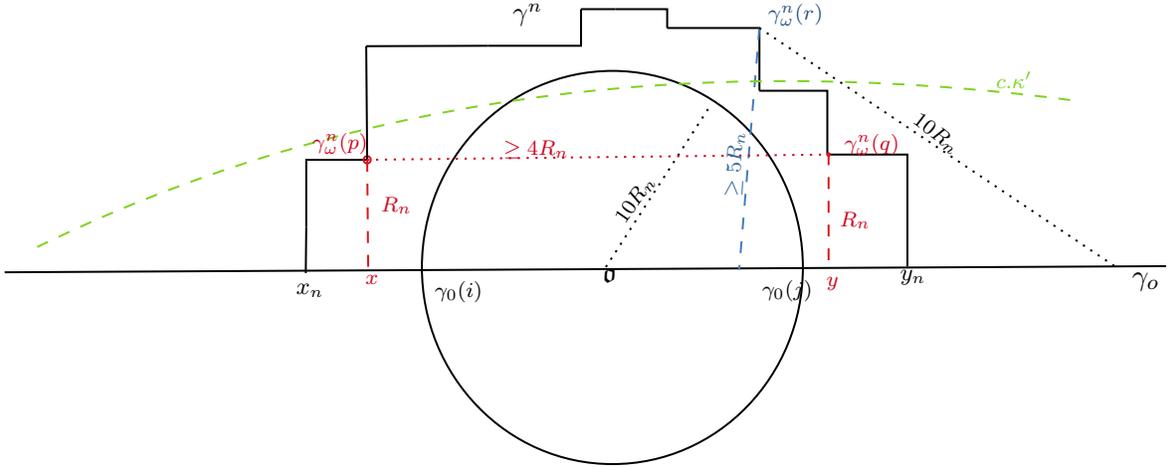
\begin{figure}

\tikzset{every picture/.style={line width=0.75pt}} 

\begin{tikzpicture}[x=0.75pt,y=0.75pt,yscale=-1,xscale=1]

\draw    (96.33,196.51) -- (682.06,193.42) ;
\draw   (248.44,173.61) -- (248.44,139.74) -- (279.17,139.74) ;
\draw  [dash pattern={on 0.84pt off 2.51pt}]  (399.55,193.86) -- (451.95,112.81) ;
\draw   (306.77,194.1) .. controls (306.77,139.29) and (349.78,94.86) .. (402.84,94.86) .. controls (455.9,94.86) and (498.91,139.29) .. (498.91,194.1) .. controls (498.91,248.9) and (455.9,293.33) .. (402.84,293.33) .. controls (349.78,293.33) and (306.77,248.9) .. (306.77,194.1) -- cycle ;
\draw    (248.44,173.61) -- (248.19,196.91) ;
\draw  [color={rgb, 255:red, 218; green, 16; blue, 40 }  ,draw opacity=1 ] (280.97,139.74) .. controls (280.97,138.71) and (280.16,137.87) .. (279.17,137.87) .. controls (278.17,137.87) and (277.36,138.71) .. (277.36,139.74) .. controls (277.36,140.77) and (278.17,141.61) .. (279.17,141.61) .. controls (280.16,141.61) and (280.97,140.77) .. (280.97,139.74) -- cycle ;
\draw [color={rgb, 255:red, 219; green, 25; blue, 25 }  ,draw opacity=1 ] [dash pattern={on 4.5pt off 4.5pt}]  (279.17,139.74) -- (279.68,196.24) ;
\draw    (278.68,82.39) -- (279.17,139.74) ;
\draw   (340,82.19) -- (387.03,82.19) -- (387.03,63.7) ;
\draw   (404.41,63.7) -- (430.48,63.7) -- (430.48,73.21) ;
\draw   (430.48,73.21) -- (476.99,73.21) -- (476.99,104.89) ;
\draw    (387.03,63.7) -- (404.41,63.7) ;
\draw   (476.99,104.89) -- (511.24,104.89) -- (511.24,137.1) ;
\draw   (511.24,137.1) -- (551.63,137.1) -- (551.63,193.6) ;
\draw  [color={rgb, 255:red, 218; green, 16; blue, 40 }  ,draw opacity=1 ] (511.24,137.1) .. controls (511.24,136.74) and (511.53,136.44) .. (511.88,136.44) .. controls (512.23,136.44) and (512.52,136.74) .. (512.52,137.1) .. controls (512.52,137.47) and (512.23,137.76) .. (511.88,137.76) .. controls (511.53,137.76) and (511.24,137.47) .. (511.24,137.1) -- cycle ;
\draw [color={rgb, 255:red, 219; green, 25; blue, 25 }  ,draw opacity=1 ] [dash pattern={on 4.5pt off 4.5pt}]  (511.88,136.44) -- (511.88,192.94) ;
\draw    (278.68,82.39) -- (340,82.19) ;
\draw [color={rgb, 255:red, 66; green, 125; blue, 190 }  ,draw opacity=1 ] [dash pattern={on 4.5pt off 4.5pt}]  (476.99,73.21) -- (466.77,195.19) ;
\draw [color={rgb, 255:red, 208; green, 2; blue, 27 }  ,draw opacity=1 ] [dash pattern={on 0.84pt off 2.51pt}]  (279.17,139.74) -- (511.24,137.1) ;
\draw  [dash pattern={on 0.84pt off 2.51pt}]  (476.99,73.21) -- (656.42,193.6) ;
\draw  [color={rgb, 255:red, 126; green, 211; blue, 33 }  ,draw opacity=1 ][dash pattern={on 4.5pt off 4.5pt}] (633.75,109.58) .. controls (438.25,84.16) and (263.64,109.36) .. (109.93,185.16) ;

\draw (241.87,201.04) node [anchor=north west][inner sep=0.75pt]  [font=\scriptsize]  {$x_{n}$};
\draw (546.53,193.91) node [anchor=north west][inner sep=0.75pt]  [font=\scriptsize]  {$y_{n}$};
\draw (401.06,168.76) node [anchor=north west][inner sep=0.75pt]  [font=\tiny,rotate=-304.84]  {$10R_{n}$};
\draw (663.5,195.08) node [anchor=north west][inner sep=0.75pt]  [font=\small]  {$\gamma _{o}$};
\draw (311.83,199.42) node [anchor=north west][inner sep=0.75pt]  [font=\tiny]  {$\gamma _{0}( i)$};
\draw (476.88,198.89) node [anchor=north west][inner sep=0.75pt]  [font=\tiny]  {$\gamma _{0}( j)$};
\draw (250.14,124.63) node [anchor=north west][inner sep=0.75pt]  [font=\tiny,color={rgb, 255:red, 208; green, 2; blue, 27 }  ,opacity=1 ]  {$\gamma _{\omega }^{n}( p)$};
\draw (518.51,125.16) node [anchor=north west][inner sep=0.75pt]  [font=\tiny,color={rgb, 255:red, 208; green, 2; blue, 27 }  ,opacity=1 ]  {$\gamma _{\omega }^{n}( q)$};
\draw (284.78,156.92) node [anchor=north west][inner sep=0.75pt]  [font=\tiny,color={rgb, 255:red, 208; green, 2; blue, 27 }  ,opacity=1 ]  {$R_{n}$};
\draw (516.12,164.34) node [anchor=north west][inner sep=0.75pt]  [font=\tiny,color={rgb, 255:red, 208; green, 2; blue, 27 }  ,opacity=1 ]  {$R_{n}$};
\draw (480.03,59.68) node [anchor=north west][inner sep=0.75pt]  [font=\tiny,color={rgb, 255:red, 33; green, 76; blue, 127 }  ,opacity=1 ]  {$\gamma _{\omega }^{n}( r)$};
\draw (457.7,159.85) node [anchor=north west][inner sep=0.75pt]  [font=\tiny,color={rgb, 255:red, 51; green, 103; blue, 161 }  ,opacity=1 ,rotate=-273.72,xslant=-0.07]  {$\geq 5R_{n}$};
\draw (346.42,128.43) node [anchor=north west][inner sep=0.75pt]  [font=\tiny,color={rgb, 255:red, 208; green, 2; blue, 27 }  ,opacity=1 ]  {$\geq 4R_{n}$};
\draw (558.09,112.87) node [anchor=north west][inner sep=0.75pt]  [font=\tiny,rotate=-37.66]  {$10R_{n}$};
\draw (276.86,196.39) node [anchor=north west][inner sep=0.75pt]  [font=\tiny,color={rgb, 255:red, 208; green, 2; blue, 27 }  ,opacity=1 ]  {$x$};
\draw (509.45,197.44) node [anchor=north west][inner sep=0.75pt]  [font=\tiny,color={rgb, 255:red, 208; green, 2; blue, 27 }  ,opacity=1 ]  {$y$};
\draw (351.17,59.32) node [anchor=north west][inner sep=0.75pt]  [font=\footnotesize]  {$\gamma ^{n}$};
\draw (396.48,192.91) node [anchor=north west][inner sep=0.75pt]  [font=\small]  {$\mathfrak{o} $};
\draw (594.86,93.68) node [anchor=north west][inner sep=0.75pt]  [font=\tiny,color={rgb, 255:red, 126; green, 211; blue, 33 }  ,opacity=1 ]  {$c.\kappa '$};

\end{tikzpicture}
\caption{Path $\gamma_n$ joining $x_n, y_n$ are outside the ball $B(\mathfrak{0},10R_n)$}
\end{figure}

Let $r$ be the first time integer such that $d\left(\gamma_{\omega}^{n}(r), \gamma_{0}([j, \infty))\right)=10 R_{n}$. Which says there exists $y \in \{\gamma_0(j),\gamma_0(j+1),..\}$ such that $d\left(\gamma_{\omega}^{n}(r), y\right)=10 R_{n}$. Now for every vertex $x \in \gamma_0(-\infty, i),$\[10R_n \leq d(\gamma^n_\omega (r),y) + d(y,x)= d(\gamma^n_\omega (r),x)\] Therefore, $d\left(\gamma_{\omega}^{n}(r), \gamma_{0}((\infty, i])\right) \geq 10 R_{n}$. and since we also have $d\left(\gamma_{\omega}^{n}(r), o\right) \geq 10 R_{n}$, we deduce that

$$
d\left(\gamma_{\omega}^{n}(r), \gamma_{0}\right) \geq 5 R_{n}
$$

If $P= max \{ k\leq r\ :d (\gamma^n_\omega(k), \gamma_0) \leq R_n\}$ and $Q= min \{ k\geq r\ : d (\gamma^n_\omega(k), \gamma_0) \leq R_n\}$ then we get $p\leq P$ and $q\geq Q$ be the largest  and the smallest integer respectively such that

$$
d\left(\gamma_{\omega}^{n}(p), \gamma_{0}\right)=d\left(\gamma_{\omega}^{n}(q), \gamma_{0}\right)=R_{n}
$$

Since we are in a connected graph, clearly for all $p \leq k \leq q$,

$$
d\left(\gamma_{\omega}^{n}(k), \gamma_{0}\right) \geq R_{n}
$$

Suppose, $x$ and $y$ are points on $\gamma_{0}$ such that $d\left(\gamma_{\omega}^{n}(p), x\right)=d\left(\gamma_{\omega}^{n}(q), y\right)=R_{n}$. \newline Therefore, \[d(x,\mathfrak{0})\geq d(\gamma_{\omega}^{n}(p),\mathfrak{0})-d(\gamma_{\omega}^{n}(p),x)\geq 9R_n \] Similarly, $d(y, o) \geq 9 R_{n}$. But since $x$ and $y$ lie on both sides of $\mathfrak{0}$, this implies that $d(x, y) \geq 18 R_{n}$ (because $\gamma_{0}$ is a geodesic). Now, by triangle inequality, we conclude that

$$
d\left(\gamma_{\omega}^{n}(p), \gamma_{\omega}^{n}(q)\right) \geq d(x, y)-2 R_{n} \geq 16 R_{n} \geq 4 R_{n}
$$

So the lemma follows.
\end{proof}

\subsection*{Proof of Theorem~\ref{thm38}}
We now let $i$ and $j$ be integers such that
\begin{equation} \label{eq1}
d\left(\gamma_{\omega}^{n}(p), \gamma_{0}(i)\right) = d\left(\gamma_{\omega}^{n}(q), \gamma_{0}(j)\right) = R_{n}.
\end{equation}
Lemma~\ref{lemma38} says, $d\left(\gamma_{\omega}^{n}(p), \gamma_{\omega}^{n}(q)\right) \geq 4R_{n}$. Hence by Lemma~\ref{lemma36}, 
\[
\left|\gamma_{\omega}^{n}([p, q])\right| \geq \phi\left(R_{n}\right)\left|\gamma_{0}([i, j])\right|.
\]
Note that, the path \(\gamma^n\) satisfies the condition \(d(\mathfrak{0}, \gamma^n) \leq |\gamma^n|\) in \(X\). Hence, applying Proposition ~\ref{prop2.13}, we get:
\[
\begin{aligned}
\left|\gamma_{\omega}^{n}([p, q])\right|_{\omega}
& \geq c \left|\gamma_{\omega}^{n}([p, q])\right| - r_{1} \\
& \geq c\, \phi\left(R_{n}\right)\left|\gamma_{0}([i, j])\right| - r_{1} \\
& = c\, \phi\left(R_{n}\right)(j - i) - r_{1}.
\end{aligned}
\]
On the other hand for the upper bound, since \(\gamma_{\omega}^{n}\) is an \(\omega\)-geodesic between \(x_{n}\) and \(y_{n}\), we have
\[
\begin{aligned}
\left|\gamma_{\omega}^{n}([p, q])\right|_{\omega}
&= d_\omega(\gamma_\omega^n(p), \gamma_\omega^n(q))\\
& \leq 2bR_{n} +2r_0+ \left|\gamma_{0}([i, j])\right|_{\omega} \quad \text{(by Line~\ref{eq1} and Proposition ~\ref{prop2.13})} \\
& \leq 2bR_{n} + b\,\left|\gamma_{0}([i, j])\right| + 3r_{0} \quad \text{(by Proposition ~\ref{prop2.13})} \\
& = 2bR_{n} + b(j - i) + 3r_{0}.
\end{aligned}
\]
By combining both lower and upper bound we have,
\[(c\, \phi\left(R_{n}\right)-b)(j-i)\leq 2bR_{n} + r_{1}+ 3r_{0}\]
Note that by triangular inequality, $j-i=\left|\gamma_{0}([i, j])\right| \geq R_{n}$. Thus,
\[(c\, \phi\left(R_{n}\right)-b)R_n\leq 2bR_{n} + r_{1}+ 3r_{0}\]
which implies, 
\begin{equation}
\phi(R_{n})\leq \frac{3b}{c} +\frac{r_{1}+ 3r_{0}}{c \cdot R_n}
\end{equation}
For $n$ large enough, this yields a contradiction since $R_n\rightarrow \infty$ and $\phi(R_n)\rightarrow \infty.$
Which implies, there exist a finite set $A:=A(\omega)$ such that $\gamma_n^\omega$ intersects $A$ in $X_\omega$ for all $n$. Therefore, in $X_\omega$ the geodesics $\gamma_n^\omega$, intersects $A$.
\newline
 Existence of bi-infinite geodesic follows from the Arzelà–Ascoli theorem since there exists at least one vertex $v\in A$ which is visited by infinitely many $\gamma_n^\omega$'s. Moreover, any ball around $v$ is finite therefore, it has finitely many paths hence by the pigeonhole principle, there is a $\omega$-geodesic which is traversed by infinitely many of the chosen $\omega$- geodesics. By taking bigger radius balls we will get a chain of these nested $\omega$-geodesics which will converge to a bi-infinite $\omega$-geodesic.

 Lastly, since $\gamma_\omega$ also intersects the set $A$, we conclude that the set distance between $\go$ and
 $\gamma_\omega$ is bounded above by the diameter of $A$.
\end{proof}

\bibliographystyle{alpha}

\end{document}